\documentclass[11pt,reqno]{amsart}
\usepackage{color}
\usepackage[english]{babel}
\usepackage{amsfonts}
\usepackage{amsmath}
\usepackage{amsthm}
\usepackage{amssymb}
\usepackage[misc]{ifsym}
\usepackage{bbm}
\usepackage{mathrsfs}
\usepackage{esint}
\usepackage{faktor}
\usepackage[utf8]{inputenc}
\usepackage{afterpage}
\usepackage[left=2.9cm,right=2.9cm,top=2.8cm,bottom=2.8cm]{geometry}
\usepackage{graphicx}
\usepackage{enumitem}
\usepackage[dvipsnames]{xcolor}
\usepackage[colorlinks=true,urlcolor=blue,citecolor=blue,linkcolor=blue,linktocpage,pdfpagelabels,bookmarksnumbered,bookmarksopen]{hyperref}

\newcommand{\doi}[1]{\url{https://doi.org/#1}}
\newcommand{\N}{\mathbb{N}}

\newcommand{\R}{\mathbb{R}}

\newcommand{\dx}{\, {\rm d} x}
\newcommand{\dy}{\, {\rm d} y}
\newcommand{\dt}{\, {\rm d} t}

\newcommand{\eps}{\varepsilon}

\renewcommand{\phi}{\varphi}

\newtheorem{lemma}{Lemma}
\newtheorem{thm}[lemma]{Theorem}
\newtheorem{prop}[lemma]{Proposition}

\theoremstyle{definition}

\newtheorem{rmk}[lemma]{Remark}

\DeclareMathOperator*{\esssup}{ess \, sup}

\begin{document}
\title[Singular fractional $p$-Laplacian problems in $\R^N$]{On a fractional $p$-Laplacian problem in the whole space and with singular reaction}
\author[L. Gambera]{L. Gambera}
\address[L. Gambera]{Dipartimento di Matematica e Informatica, Universit\`a degli Studi di Catania, Viale A. Doria 6, 95125 Catania, Italy}
\email{laura.gambera10@gmail.com}
\author[S.A. Marano]{Salvatore A. Marano}
\address[S.A. Marano]{Dipartimento di Matematica e Informatica, Universit\`a degli Studi di Catania, Viale A. Doria 6, 95125 Catania, Italy}
\email{marano@dmi.unict.it}
\begin{abstract}
The existence of positive, pointwise decaying at infinity, weak solutions to a fractional $p$-Laplacian problem in the whole space and with singular reaction is established. Truncation arguments, variational methods, as well as suitable a priori estimates are exploited.
\end{abstract}
\let\thefootnote\relax
\footnote{\Letter\quad Corresponding author: Laura Gambera (laura.gambera10@gmail.com).}
\maketitle
\section{Introduction}
Nonlocal differential operators, like the fractional $p$-Laplacian of order $s$, usually denoted with $(-\Delta_p)^s$, have been introduced in recent years to model a variety of phenomena, including anomalous diffusion and processes characterized by long-range interactions. Due to its broad applicability, they have attracted considerable attention, also in connection with problems arising from game theory, finance, image processing, materials science, etc. This growing interest has led to a wide literature that addresses existence, uniqueness, and multiplicity of solutions for equations with such operators. Let us mention \cite{IPS,MM,FI,CD,BI,DHR,GMM2,CMZ}, as well as the survey \cite{DNPV} for a comprehensive account of underlying fractional Sobolev spaces. A significant progress has also been made in establishing regularity results; see, for instance, \cite{IMS,IMS2,CPT}.

The existence of solutions to singular equations driven by nonlocal differential operators has recently been investigated. For instance, Dirichlet fractional $p$-Laplacian problems are studied in \cite{CMSS}, while \cite{G,GM} deal with the fractional $(p,q)$-Laplacian. It is worth noting that all these works are set in a bounded domain $\Omega\subseteq\R^N$. On the contrary, here, we consider the problem
\begin{equation}\label{prob}\tag{P}
\left\{
\begin{alignedat}{2}
(-\Delta_p)^s u & =f(x,u)\;\; && \mbox{in}\;\;\R^N,\\
u & >0 && \mbox{in}\;\;\R^N,\\
u(x) & \to 0 && \mbox{as}\;\; |x|\to +\infty,
\end{alignedat}
\right.
\end{equation}
where $N\ge 3$, $0<s<1$, $2\le p<\frac{N}{s}$. while $(-\Delta_p)^s$ is formally defined by
$$(-\Delta_p)^s u(x):=2\lim_{\eps\to 0^+}\int_{\R^N\setminus B_\eps(x)}
\frac{|u(x)-u(y)|^{p-2}(u(x)-u(y))}{|x-y|^{N+sp}}\dy,\quad x\in\R^N.$$
The reaction term $f:\mathbb{R}^N\times\mathbb{R}^+\to \mathbb{R}_0^+$ satisfies Carath\'eodory's conditions. Morever,
\begin{equation}\label{hypf}
\tag{$\mathrm{H_f}$}
a(x)t^{-\gamma}\le f(x,t)\le a(x)\big(t^{-\gamma}+t^r\big)\quad\forall\,(x,t)\in\R^N\times\R^+,
\end{equation}
with $0<\gamma<1<r<p-1$ and $a:\R^N\to\R^+_0$ measurable function such that
\begin{equation}\label{hypa}
\tag{$\mathrm{H_a}$}
0\le a(x)\le c_a w(x),\quad\text{where}\quad  w(x):=\frac{1}{1+|x|^{N+\alpha}},\;\; x\in\R^N,
\end{equation}
for some $c_a>0$, $\alpha\in\left(\gamma\frac{N-sp}{p-1},\gamma\frac{N-sp}{p-1}+sp \right)$.

If $D^{s,p}(\R^N)$ denotes the completion of $C^\infty_c(\R^N)$ with respect to Gagliardo's norm (cf. Section \ref{prelim}) then a function $u\in D^{s,p}(\R^N)$ is called a \textit{weak solution} of \eqref{prob}  provided 
\begin{equation*}
\int_{\R^{2N}}\frac{|u(x)-u(y)|^{p-2}(u(x)-u(y))(\phi(x)-\phi(y))}{|x-y|^{N+sp}}\dx\dy = \int_{\R^N}f(\cdot,u)\phi\dx
\end{equation*}
for all $\phi\in D^{s,p}(\R^N)$, $u>0$ in $\R^N$, and $\displaystyle{\lim_{|x|\to\infty}} u(x)=0$. Our main result, namely Theorem \ref{exres} below, ensures that, under conditions \eqref{hypf} and \eqref{hypa}, there exists a weak solution to problem \eqref{prob}. 

From a technical point of view, there are at least four noteworthy features:
\begin{itemize}
\item[a)] the driving differential operator is nonlocal;
\item[b)] $f(x,\cdot)$ turns out to be singular at zero, which means $\displaystyle{\lim_{t\to 0^+}} f(x,t)=+\infty$;
\item[c)] one seeks solutions defined in the whole $\R^N$;
\item[d)] a pointwise decay of solutions at infinity is reqired.
\end{itemize}
The works \cite{BMS,CPT} play a significant role to get d). In fact, \cite{CPT} provides global $L^\infty$-bounds and decay estimates for solutions $u\in D^{s,p}(\R^N)$ of fractional $p$-Laplacian equations, while \cite{BMS} contains a rigorous computation of the fractional $p$-Laplacian fundamental solution and establishes the sharp asymptotic behavior at infinity of extremal functions for the fractional critical Sobolev embedding.

The paper is organized as follows. Section \ref{prelim} contains preliminary facts. In the next one, we focus on the pure singular problem and, via a shifting method, show that it admits a weak solution. Finally, Section 4 is devoted to the study of \eqref{prob}, which we solve by means of a truncation argument, combined with a variational approach and suitable a priori estimates.
\section{Basic definitions and auxiliary results}\label{prelim}
Given two topological spaces $X$ and $Y$, the symbol $X\hookrightarrow Y$ means that $X$ continuously embeds in $Y$. 

Let $X$ be a real Banach space with topological dual $X^*$ and duality brackets $\langle\cdot,\cdot\rangle$. A function $A:X\to X^*$ is called \emph{strictly monotone} when $\langle A(x)-A(z),x-z\rangle>0$ for all $x,z\in X$ with $x\neq z$.

If $r>0$ and $x_0\in\R^N$ then $B_r(x_0):=\{ x\in\R^N: |x-x_0|<r\}$, $B_r:=B_r(0)$, $\overline{B}_r$ denotes the closure of $B_r$, $B_r^c:=\R^N\setminus B_r$, $|E|$ indicates the $N$-dimensional Lebesgue measure of $E\subseteq \mathbb{R}^{N}$, $\chi_E$ stands for the characteristic function of $E$, 
$$t_\pm:=\max\{\pm t,0\},\;\; t\in\R,$$
while $C$, $C_1$, etc. are positive constants, which may change value from line to line, whose dependencies will be specified when necessary. 

Let $X(\R^N)$ be a real-valued function space on $\R^N$ and let $u,v\in X(\R^N)$. We simply write $u\leq v$ when $u(x)\leq v(x)$ a.e. in $\R^N$. Analogously for $u<v$, etc. To shorten notation, define
\begin{equation*}
\{u\leq v\}:=\{x\in\R^N:u(x)\leq v(x)\},\quad  X(\R^N)_+:=\{w\in X(\R^N): w\ge 0\}.
\end{equation*}
The meaning of $\{u<v\}$ and so on is similar. Henceforth, $q'$ indicates the conjugate exponent of $q\ge 1$, while
\begin{equation*}
\Vert v\Vert_q:=\left\{ 
\begin{array}{ll}
\left(\int_{\R^N}|v(x)|^q\dx\right)^{\frac{1}{q}} & \text{ if }1\leq q<\infty, \\ 
\phantom{} &  \\ 
\underset{x\in\R^N}{\esssup}\, |v(x)| & \text{ when } q=\infty.
\end{array}
\right.
\end{equation*}
Given $s\in(0,1)$ and $p\in[1,+\infty)$, the Gagliardo semi-norm of a measurable function $u:\R^N\to\R$ is
\begin{equation*}
[u]_{s,p}:=\left(\int_{\R^N\times\R^N}\frac{|u(x)-u(y)|^p}{|x-y|^{N+sp}}\dx\dy \right)^{\frac{1}{p}}.
\end{equation*}
If $1<p<\frac{N}{s}$, write (when no confusion can arise)
\begin{equation*}
p^*:=\frac{Np}{N-sp}   
\end{equation*}
and denote by $D^{s,p}(\R^N)$ the homogeneous fractional Sobolev space
\begin{equation*}
D^{s,p}(\R^N):= \left\{u\in L^{p^*}(\R^N):\ [u]_{s,p}<+\infty \right\},
\end{equation*}
endowed with the norm
$$\| u\|_{s,p}:=[u]_{s,p},\;\; u\in D^{s,p}(\R^N).$$
It is known \cite{BGV} that $(D^{s,p}(\R^N),\Vert\cdot\Vert_{s,p})$ turns out a uniformly convex (hence, reflexive) Banach space.

The operator $(-\Delta_p)^s:D^{s,p}(\R^N)\to (D^{s,p}(\R^N))^*$ defined by setting
\begin{equation*}
\langle(-\Delta_p)^s u,v\rangle:=
\int_{\R^N\times\R^N}\frac{|u(x)-u(y)|^{p-2}(u(x)-u(y))(v(x)-v(y))}{|x-y|^{N+sp}}\dx\dy
\end{equation*}
for every $u,v\in D^{s,p}(\R^N)$ is called fractional $p$-Laplacian of order $s$. It enjoys the next useful property.
\begin{lemma}\label{monfraclapl}
Suppose $2\le p<\frac{N}{s}$. Then $(-\Delta_p)^s$ turns out to be strictly monotone.
\end{lemma}
\begin{proof}
For every $a,b\in\R$ one has
$$\left(|a|^{p-2}a -|b|^{p-2}b\right) (a-b) \ge C_p |a-b|^p,$$
with appropriate $C_p>0$; see, e.g., \cite[Example 25.5]{Z}. If $u,v\in D^{s,p}(\R^N)$ then choosing $a:=u(x)-u(y)$, $b:=v(x)-v(y)$, dividing by $|x-y|^{N+sp}$, and integrating on $\R^N\times\R^N$ easily produces
$$\langle (-\Delta_p)^s u-(-\Delta_p)^sv, u-v\rangle\ge C_p\Vert u-v\Vert_{s,p}^p,$$
as desired. 
\end{proof}
Let $w:\R^N\to\R^+_0$ be as in \eqref{hypa}, namely
\begin{equation*}
w(x):=\frac{1}{1+|x|^{N+\alpha}}\;\;\forall\, x\in\R^N,   
\end{equation*}
where $\alpha>0$ and let $q\geq 1$. On the linear space space
$$L^q(\R^N,w):=\left\{u:\R^N\to\R \; \text{measurable}: \int_{\R^N} w|u|^q \dx<\infty \right\},$$
we will consider the norm
$$\Vert u\Vert_{q,w}=\left(\int_{\R^N} |u|^q w\dx\right)^{\frac{1}{q}}.$$
If $q>1$ then $(L^q(\R^N,w),\Vert\cdot\Vert_{q,w})$ turns out reflexive and separable.
Moreover (cf. \cite[Lemma 3.3]{CPT}),
\begin{lemma}\label{weightedemb}
Assume $2\leq p<\frac{N}{s}$. Then:
\begin{itemize}
\item[{\rm (a)}] $D^{s,p}(\R^N)\hookrightarrow L^q(\R^N,w)$ whatever $1\le q\le p^*$.
\item[{\rm (b)}] The embedding in {\rm (a)} is compact once $q<p^*$.
\end{itemize}
\end{lemma}
Given any open set $\Omega\subseteq\R^N$, we define
\begin{equation*}
\begin{split}
\tilde{D}^{s,p}(\Omega):=\Big\{ u\in & L^{p-1}_{\rm loc}(\R^N)\cap L^{p^*}(\Omega): \exists E\supseteq\Omega \;\text{with $E^c$ compact,}\; {\rm dist}(E^c,\Omega)>0,\\
& \text{and}\;[u]_{W^{s,p}(E)}<\infty\Big\}.
\end{split}
\end{equation*}
Here,
\begin{equation*}
[u]_{W^{s,p}(E)}:=
\left(\int_{E\times E}\frac{|u(x)-u(y)|^p}{|x-y|^{N+sp}}\,\dx\dy\right)^\frac{1}{p}.  
\end{equation*}
The following results, established in Theorem 2.7, Lemma A.2, and Propositions 3.5--3.6 of \cite{BMS}, will play a key role in the proof of Theorem \ref{exres}.

\begin{thm}\textnormal{(Comparison principle in general domains)}\label{weakcomp} 
Let $\Omega\subseteq\R^N$ be an open set. If $u,v\in\tilde{D}^{s,p}(\Omega)$ satisfy
\begin{equation*}
u\leq v\;\;\mbox{in}\;\;\Omega^c\quad \text{and}\quad (-\Delta_p)^s u\le (-\Delta_p)^s v\;\; \text{in} \;\;\Omega
\end{equation*}
then $u\le v$ in $\Omega$.
\end{thm}
\begin{thm}\label{fundsol}
Suppose $0<\frac{N-sp}{p}<\beta<\frac{N}{p-1}$. Then, for any $R>0$, the function $v_\beta(x):=|x|^{-\beta}$ fulfills
\begin{equation*}
(-\Delta_p)^s v=C(\beta)|x|^{-\beta(p-1)-sp}\;\;\text{weakly in}\;\;\R^N\setminus \overline{B}_R,
\end{equation*}
where $C(\beta)$ is given by
\[
C(\beta):=2\int_0^1\rho^{sp-1}\left[1-\rho ^{N-sp-\beta(p-1)}\right] |1-\rho^\beta|^{p-1} \Phi(\rho)\,{\rm d}\rho
\]
while
\[
\Phi(\rho):=\mathcal{H}^{N-2}\!\left(S^{N-2}\right)
\int_{-1}^{1}\frac{(1-t^2)^{\frac{N-2}{2}}}{(1-2t\rho+\rho^2)^{\frac{N+sp}{2}}}\,{\rm d}t.
\]
\end{thm}
\begin{rmk}
If $\beta=\frac{N-sp}{p-1},$ then $C(\beta)=0$. Accordingly, the function $v_\beta(x):= |x|^{-\frac{N-sp}{p-1}}$ belongs to $\tilde{D}^{s,p}(B_R^c)$ and solves $(-\Delta_p)^s v= 0$ in $\R^N\setminus\overline{B}_R$.
\end{rmk}

\begin{prop}\label{decay}
Let $1<p<\frac{N}{s}$. Then, for every $R>0$, there exists a unique, radial, non-increasing function $u_R\in D^{s,p} (\R^N)$ such that 
\begin{equation*}
\left\{
\begin{alignedat}{2}
(-\Delta_p)^s u_R & =0 \quad && \text{in}\;\;\R^N\setminus\overline{B}_R,\\
u_R & \equiv1 && \text{in}\;\;\overline{B}_R.
\end{alignedat}
\right.
\end{equation*}
Moreover, when $R=1$, the following decay estimate holds:
\begin{equation}\label{decayu1}
\frac{1}{C}|x|^{-\frac{N-sp}{p-1}}\le u_1(x)\le p^{\frac{1}{p-1}}|x|^{-\frac{N-sp}{p-1}}\;\;\forall\, x\in B_1^c,
\end{equation}
where $C:=C(N,p,s)>1.$
\end{prop}
Finally, we point out the next weak Harnack-type inequality \cite[Theorem 5.2]{IMS} for non-negative super-solutions.

Let $\Omega\subseteq\R^N$ be bounded. We set
\begin{equation*}
\tilde{W}^{s,p}(\Omega):=\left\{u\in L^{p}_{\rm loc}(\R^N): \exists E\Supset\Omega : \Vert u\Vert_{L^p(E)}+[u]_{W^{s,p}(E)}
+\int_{\R^N}\frac{|u(x)|^{p-1}}{(1+|x|)^{N+sp}}{\rm d}x<\infty\right\}.
\end{equation*}
\begin{thm}\label{harnack}
If $u\in\tilde{W}^{s,p}(B_{R/3})$ satisfies
\begin{equation*}
\left\{
\begin{alignedat}{2}
(-\Delta)^{s}_p u &\ge K \quad && \text{in }B_{R/3},\\
u & \ge0 && \text{in }\R^N,
\end{alignedat}
\right.
\end{equation*}
where $K\ge0$, then there exist $\sigma\in(0,1)$, $\tilde{C}>0$ such that 
\begin{equation*}
\inf_{B_{R/4}} u\ge\sigma\left(\fint_{B_R\setminus B_{R/2}}
u^{p-1}\dx\right)^{\frac{1}{p-1}}-\tilde{C}(KR^{ps})^{\frac{1}{p-1}}.
\end{equation*}
\end{thm}
\section{A pure singular problem}
In this section, we focus on the pure singular problem
\begin{equation}\label{singprob}
\tag{$\bar{\rm P}$}
\left\{
\begin{alignedat}{2}
(-\Delta_p)^s u & =a(x)u^{-\gamma}\;\; && \mbox{in}\;\;\R^N,\\
u & >0 && \mbox{in}\;\;\R^N,\\
u(x) & \to 0 && \mbox{as}\;\; |x|\to\infty,
\end{alignedat}
\right.
\end{equation}
and prove that it admits a weak solution $\bar{u}\in D^{s,p}(\R^N)$. With this aim, likewise \cite{CMSS}, let us first study the regularized problem
\begin{equation}\label{regprob}
\tag{${\rm P}_n$}
\left\{
\begin{alignedat}{2}
(-\Delta_p)^s u & =a(x)\left(u_+ +\frac{1}{n}\right)^{-\gamma}\;\; && \mbox{in}\;\;\R^N,\\
u & \ge 0 && \mbox{in}\;\;\R^N,
\end{alignedat}
\right.
\end{equation}
where $n\in\N$. Define, for very $(x,t)\in\R^N\times\R$,
$$A(x,t):=\int_0^t a(x)\left(\tau_++\frac{1}{n}\right)^{-\gamma}{\rm d}\tau.$$
One evidently has
\begin{equation}\label{new}
A(x,t)\leq\int_0^{|t|} a(x)\left(\tau_+ +\frac{1}{n}\right)^{-\gamma}{\rm d}\tau
\le\int_0^{|t|} a(x)\tau^{-\gamma}{\rm d}\tau=a(x)\frac{|t|^{1-\gamma}}{1-\gamma},
\end{equation}
while the energy functional arising from \eqref{regprob} is
\begin{equation*}
J_n(u):=\frac{1}{p}\Vert u\Vert_{s,p}^p-\int_{\R^N} A(\cdot,u)\dx,\quad u\in D^{s,p}(\R^N).
\end{equation*}
\begin{lemma}\label{functprop}
If \eqref{hypa} holds then $J_n$ is of class $C^1$ in $D^{s,p}(\R^N)$, weakly sequentially lower semi-continuous, and coercive.
\end{lemma}
\begin{proof}
From $a\in L^1(\mathbb{R}^N)\cap L^{\infty}(\mathbb{R}^N)$ it follows $a\in L^r(\mathbb{R}^N)$ whatever $r>1$; see, e.g., \cite[Proposition A.6]{GG2}. Hence, thanks to H\"older's inequality,
\begin{equation*}
u\in D^{s,p}(\R^N)\implies\int_{\R^N} a(x)|u(x)|^{1-\gamma}\dx
\le C\|a\|_{\frac{p^*}{p^*-1+\gamma}}\|u\|_{p^*}^{1-\gamma}< \infty.
\end{equation*}
This means that $J_n$ is well defined. Further, $J_n\in C^1(D^{s,p}(\R^N))$, as a standard argument shows. Now, let $u_m\rightharpoonup u$ in $D^{s,p}(\R^N)$. The same reasoning made for proving \cite[Lemma 3.3]{CPT} ensures here that to each $R>0$ there corresponds a sub-sequence, again denoted by $\{u_m\}$, fulfilling $u_m(x)\to u(x)$ a.e. in $B_R$. Through a diagonalization procedure we next arrive at $u_m(x)\to u(x)$ for almost all $x\in\R^N$. Due to Fatou's lemma and \eqref{new} one thus has
\begin{equation*}
\limsup_{m\to+\infty}\int_{\R^N}A(x,u_m(x))\dx\le\int_{\R^N} A(x,u(x))\dx\, .    
\end{equation*}
On the other hand, Proposition 3.5 of \cite{B} entails
$$\liminf_{m\to\infty}\Vert u_m\Vert_{s,p}^p\ge\Vert u\Vert_{s,p}^p.$$
Hence,
\begin{equation*}
\begin{aligned}
\liminf_{m\to \infty} J_n(u_m) & \ge\liminf_{m\to\infty}\frac{1}{p}\Vert u_m\Vert_{s,p}^p 
-\limsup_{m\to\infty}\int_{\R^N} A(x,u_m(x))\dx\\
& \ge\frac{1}{p}\Vert u\Vert_{s,p}^p-\int_{\R^N} A(x,u(x))\dx =J_n(u),
\end{aligned}
\end{equation*}
i.e., $J_n$ is weakly sequentially lower semi-continuous. Let us finally come to coercivity. Via Young's inequality with $\eps$ and (a) of Lemma \ref{weightedemb} we obtain  
\begin{equation*}
\begin{aligned}
\int_{\mathbb{R}^N} a(x)|u(x)|^{1-\gamma}\dx 
& =\int_{\mathbb{R}^N} a(x)^{\frac{p-1+\gamma}{p}}a(x)^{\frac{1-\gamma}{p}}|u(x)|^{1-\gamma}\dx\\
& \le\eps\int_{\mathbb{R}^N} a(x) |u(x)|^{p}\dx +C_\eps\|a\|_1\\
& \le C\eps\Vert u\Vert_{s,p}^p+ C_\varepsilon\|a\|_1\, ,
\end{aligned}
\end{equation*}
whence
\begin{equation*}
J_n(u)\ge\left(\frac{1}{p}-C\eps\right)\Vert u\Vert_{s,p}^p-(C_\eps+C)\|a\|_1\quad\forall\, u\in D^{s,p}(\R^N). 
\end{equation*}
Choosing $\eps>0$ sufficiently small yields $\displaystyle{\lim_{\Vert u\Vert_{s,p}\to\infty}} J_n(u)=+\infty$, as desired. 
\end{proof}
\begin{thm}\label{regsol}
Under hypothesis \eqref{hypa}, for every $n\in \N$, problem \eqref{regprob} possesses a solution $u_n\in D^{s,p}(\R^N)_+$. The sequence $\{ u_n\}$ is bounded in $D^{s,p}(\R^N)$.
\end{thm}
\begin{proof}
By Lemma \ref{functprop}, the Weierstrass-Tonelli theorem \cite[Theorem 1.2]{S} can be applied to $J_n$. So, there exists $u_n\in D^{s,p}(\R^N)$ such that
$$J_n(u_n)=\inf_{u\in D^{s,p}(\R^N)} J_n(u)\, .$$
In particular, $u_n$ weakly solves the differential equation of \eqref{regprob}, i.e.,
\begin{equation}\label{wsolrp}
\begin{split}
\int_{\R^{2N}} & 
\frac{|u_n(x)-u_n(y)|^{p-2}(u_n(x)-u_n(y))(v(x)-v(y))}{|x-y|^{N+sp}}\dx\dy\\
& =\int_{\R^N} a\left((u_n)_++\frac{1}{n}\right)^{-\gamma} v\dx
\quad\forall\, v\in D^{s,p}(\R^N). 
\end{split}
\end{equation} 
We will show that $u_n\geq 0$. Testing \eqref{wsolrp} with $v:=-(u_n)_-$ and using the inequality 
\begin{equation*}
\pm\vert a-b\vert^{p-2}(a-b)\bigl(a_{\pm}-b_{\pm}\bigr)
\ge\lvert a_{\pm}-b_{\pm}\rvert^{p},\quad a,b\in\R
\end{equation*}
(cf. \cite[p. 89]{MS}) produces
\begin{equation*}
\begin{split}
\Vert(u_n)_-\Vert_{s,p}^p & \le\int_{\R^{2N}} 
-\frac{|u_n(x)-u_n(y)|^{p-2}(u_n(x)-u_n(y))((u_n)_-(x)-(u_n)_-(y))}{|x-y|^{N+sp}}\dx\dy\\
& =-\int_{\R^N} a\left((u_n)_+ +\frac{1}{n}\right)^{-\gamma}(u_n)_- \dx
=-n^\gamma\int_{\R^N} a\, (u_n)_-\dx\leq 0,
\end{split}
\end{equation*}
which obviously forces $u_n=(u_n)_+\geq 0$. Finally, via \eqref{wsolrp} written for $v:=u_n$, H\"older's inequality, besides the continuous embedding $D^{s,p}(\R^N) \hookrightarrow L^{p^*}(\R^N)$, we obtain
\begin{equation*}
\Vert u_n\Vert_{s,p}^p
\le\int_{\R^N} a\,|u_n|^{1-\gamma}\dx
\le\|a\|_{\frac{p^*}{p^*-1+\gamma}}\,\|u_n\|_{p^*}^{1-\gamma}
\le C\|a\|_{\frac{p^*}{p^*-1+\gamma}}\,\Vert u_n\Vert_{s,p}^{1-\gamma}\quad\forall\, n\in\N.
\end{equation*}
The boundedness of $\{u_n\}$ now follows from $p>1-\gamma>0$. 
\end{proof}
Since $D^{s,p}(\R^N)$ is reflexive, there exists $\overline{u}\in D^{s,p}(\R^N)$ such that, up to subsequences,
\begin{equation}\label{convergence}
u_n \rightharpoonup\overline{u}\;\;\text{in}\;\; D^{s,p}(\R^N),\quad
u_n(x)\to\overline{u}(x)\;\;\mbox{for almost every $x\in\R^N$;}
\end{equation}
see the proof of Lemma \ref{functprop}. Moreover,
\begin{lemma}\label{monseq}
The sequence $\{u_n\}$ given by Theorem \ref{regsol} turns out monotone non-decreasing.
\end{lemma}
In fact, on the set $\{u_n>u_{n+1}\}$  one evidently has 
\begin{equation*}
(-\Delta_p)^s u_n=a\left(u_n+\frac{1}{n}\right)^{-\gamma}
\le a\left(u_{n+1}+\frac{1}{n+1}\right)^{-\gamma} =(-\Delta_p)^s u_{n+1}. 
\end{equation*}
Consequently, likewise \cite[Lemma 2.8]{BG}, the conclusion is achieved once we show that
\begin{lemma}
If $v,w\in D^{s,p}(\R^N)$ and $\langle(-\Delta_p)^s w,\phi\rangle\le\langle(-\Delta_p)^s v,\phi\rangle$ for every $\phi\in D^{s,p}(\R^N)_+$ fulfilling $\phi\equiv 0$ on $\{w\le v\}$  then $w\le v$ in $\R^N$.
\end{lemma}
\begin{proof}
To simplify formulae, define
\begin{equation*}
V(x,y):=\frac{|v(x)-v(y)|^{p-2}(v(x)-v(y))}{|x-y|^{N+sp}},\quad
W(x,y):=\frac{|w(x)-w(y)|^{p-2}(w(x)-w(y))}{|x-y|^{N+sp}}\, .
\end{equation*}
and pick $\phi:=(w-v)_+$. Using the assumptions we get
\begin{equation*}
\begin{aligned}
0 & \le\int_{\R^{2N}} V(x,y)(\phi(x)-\phi(y))\dx\dy
-\int_{\R^{2N}} W(x,y)(\phi(x)-\phi(y))\dx\dy\\
& =\int_{\{w\ge v\}}\dy\int_{\{w\ge v\}}(V(x,y)-W(x,y))(w(x)-v(x)-w(y)+v(y))\dx\\
& \hskip1.5cm+\int_{\{w<v\}}\dy\int_{\{w>v\}} (V(x,y)-W(x,y))\phi(x)\dx\\
& \hskip1.5cm-\int_{\{w>v\}}\dy\int_{\{w<v\}} (V(x,y)-W(x,y))\phi(y)\dx.
\end{aligned}
\end{equation*}
Since $L(a,b):=|a-b|^{p-2}(a-b)$, $(a,b)\in\R^2$, is increasing with respect to $a$ and decreasing in $b$, this entails
\begin{equation*}
\begin{aligned}
0 & \le\int_{\R^{2N}} (V(x,y)-W(x,y))(\phi(x)-\phi(y))\dx\dy\\
& =-\int_{\{w\ge v\}}\dy\int_{\{w\ge v\}}(V(x,y)-W(x,y))(v(x)-v(y)-w(x)+w(y))\dx\\
& \hskip1.5cm+\int_{\{w<v\}}\dy\int_{\{w>v\}}(W(x,y)-W(x,y))\phi(x)\dx\\
& \hskip1.5cm-\int_{\{w>v\}}\dy\int_{\{w<v\}}(W(x,y)-W(x,y))\phi(y)\dx\\
& =-\int_{\{w\ge v\}\times\{w\ge v\}}(V(x,y)-W(x,y))(v(x)-v(y)-w(x)+w(y))\dx\dy\le 0
\end{aligned}
\end{equation*}
by Lemma \ref{monfraclapl}. Therefore,
\begin{equation*}
\int_{\{w\ge v\}\times\{w\ge v\}} (V(x,y)-W(x,y))(\phi(x)-\phi(y))\dx\dy=0,   
\end{equation*}
which, due to the choice of $\varphi$, implies $w\le v$ on the whole $\R^N$.
\end{proof}
\begin{thm}\label{overusol}
Let \eqref{hypa} be satisfied and let $\overline{u}$ be as in \eqref{convergence}. Then $\overline{u}$ weakly solves \eqref{singprob}.      
\end{thm}
\begin{proof}
Theorem \ref{regsol} ensures that \eqref{wsolrp} holds for every $n\in\N$, $v\in D^{s,p}(\R^n)$. Lemma \ref{monseq} and the arguments adopted in the proof of \cite[Lemma 2.9]{BG} produce
\begin{equation}\label{convrhs1}
\lim_{n\to+\infty}\int_{\R^N} a\left(u_n+\frac{1}{n}\right)^{-\gamma}v\dx=  
\int_{\R^N} a\,\overline{u}^{-\gamma}v\dx.
\end{equation}
We next claim that
\begin{equation}\label{convlhs1}
\lim_{n\to+\infty}\int_{\R^{2N}} U_n(x,y)(v(x)-v(y))\dx\dy
=\int_{\R^{2N}} \overline{U}(x,y)(v(x)-v(y))\dx\dy,
\end{equation}
where, to shorten formulae,
\begin{equation*}
\begin{aligned}
& U_n(x,u):=\frac{|u_n(x)-u_n(y)|^{p-2}(u_n(x)-u_n(y))}{|x-y|^{N+sp}},\\
&\overline{U}(x,y):=\frac{|\overline{u}(x)-\overline{u}(y)|^{p-2}(\overline{u}(x)-\overline{u}(y))}{|x-y|^{N+sp}}.
\end{aligned}
\end{equation*}
In fact, by Theorem \ref{regsol} again, the sequence
$\left\{|x-y|^{\frac{N+sp}{p}} U_n(x,y)\right\}$
turns out bounded in $L^{p'}(\mathbb{R}^{2N})$, while \eqref{convergence} entails
$$\lim_{n\to+\infty}|x-y|^{\frac{N+sp}{p}} U_n(x,y)=|x-y|^{\frac{N+sp}{p}}\,\overline{U}(x,y)\;\;\mbox{for almost every}\;(x,y)\in\R^{2N}.
$$
So, taking a sub-sequence if necessary, 
$$|x-y|^{\frac{N+sp}{p}} U_n(x,y)\rightharpoonup |x-y|^{\frac{N+sp}{p}}\,\overline{U}(x,y)\;\;\mbox{in}\; L^{p'}(\R^{2N}),$$
which yields \eqref{convlhs1}, because the function $|x-y|^{-\frac{N+sp}{p}}(v(x)-v(y))$ belongs to $L^{p}(\R^{2N})$, as $v\in D^{s,p}(\R^N)$. Now, through \eqref{convrhs1}--\eqref{convlhs1} we see that $\overline{u}$ weakly solves the equation 
\begin{equation}\label{basiceq}
(-\Delta_p)^s u= a(x) u^{-\gamma}\;\;\mbox{in}\;\;\R^N.    
\end{equation}
Since one evidently has $\overline{u}\ge 0$, the conclusion will follow from the next lemma.
\end{proof}
\begin{lemma}
If $\overline{u}\in D^{s,p}(\R^N)_+$ is a solution of \eqref{basiceq} then 
\begin{equation}\label{decayest}
\tilde{c}|x|^{-\frac{N-sp}{p-1}}\le \overline{u}(x)\le D|x|^{\frac{1}{p-1}\left(-N-\alpha+\gamma  \frac{N-sp}{p-1}-sp\right)}\quad\text{for almost every $x\in B_1^c$},    
\end{equation}
with $\tilde{c}, D >0$ suitable constants. So, $\overline{u}(x)\to 0$ as $|x|\to\infty$.
\end{lemma}
\begin{proof}
By Theorem \ref{harnack}, written for $K:=0$ and $R:=4$, one has
\begin{equation*}
\inf_{B_1} \overline{u}\ge\sigma\left(\fint_{B_4\setminus B_2} \overline{u}^{p-1}\dx\right)^{\frac{1}{p-1}}
=\frac{\sigma}{[\omega_N\left(4^N-2^N\right)]^{\frac{1}{p-1}}}
\left(\int_{B_4\setminus B_2} \overline{u}^{p-1}\dx\right)^{\frac{1}{p-1}}.
\end{equation*}
Hence,
\begin{equation}\label{lowest}
\inf_{B_1} \overline{u}\ge\sigma M,   
\end{equation}
where $\sigma\in (0,1)$, $M:=M(N,p,\| u\|_{L^p(B_4)})>0$. Let $u_1\in D^{s,p}(\R^N)$ be as in Proposition \ref{decay}, with $R:=1$, and let $0<c<\sigma M$. The function $\tilde{u}_1:= c u_1$ enjoys the properties
\begin{equation*}
(-\Delta_p)^s\tilde{u}_1=0\;\;\text{in}\;\;\R^N\setminus\overline{B}_1,\quad
\tilde{u}_1\equiv c\;\;\text{on}\;\;\overline{B}_1.
\end{equation*}
Moreover, thanks to \eqref{lowest},
$$\tilde{u}_1\le \overline{u}\;\;\text{in}\;\;\overline{B}_1,\qquad
(-\Delta_p)^s\tilde{u}_1
\le(-\Delta_p)^s\overline{u}\;\;\text{in}\;\;\R^N\setminus\overline{B}_1.$$
Because of Theorem \ref{weakcomp} and \eqref{decayu1}, this implies
\begin{equation}\label{lowestu}
\tilde{c}|x|^{-\frac{N-sp}{p-1}}\le\tilde{u}_1(x)\le\overline{u}(x)\quad\forall\, x\in B_1^c,    
\end{equation}
where $\tilde{c}:=\tilde{c}(N,p,s,C,c)>0$. Via \eqref{hypa} and \eqref{lowestu} we next obtain
\begin{equation}\label{auest}
a(x)\overline{u}(x)^{-\gamma}
\le\frac{c_a}{1+|x|^{N+\alpha}}\tilde{c}^{-\gamma}|x|^{\gamma\frac{N-sp}{p-1}}
\le\frac{\hat{c}}{|x|^{N+\alpha-\gamma\frac{N-sp}{p-1}}},\quad x\in B_1^c,
\end{equation}
for some $\hat{c}>0$. Putting \eqref{hypa}, \eqref{lowest}, and \eqref{auest} together gives $a\, \overline{u}^{-\gamma}\in L^1(\R^N)\cap L^\infty(\R^N)$. In fact,
\begin{equation*}
\begin{split}
\int_{\R^N} a\, \overline{u}^{-\gamma}\dx\le\|a\|_\infty \left(\inf_{B_1} \overline{u}\right)^{-\gamma}|B_1|
+\hat c\int_{B_1^c}\frac{1}{|x|^{N+\alpha-\gamma\frac{N-sp}{p-1}}}\dx<\infty,\\
\sup_{x\in\R^N}a(x)\overline{u}(x)^{-\gamma}\le\|a\|_\infty \left(\inf_{B_1} \overline{u}\right)^{-\gamma}+
\hat c\,\sup_{x\in B_1^c}\frac{1}{|x|^{N+\alpha-\gamma\frac{N-sp}{p-1}}}<\infty.
\end{split}
\end{equation*}
Consequently,
\begin{equation}\label{sumau}
a\, \overline{u}^{-\gamma}\in L^\tau(\R^N)\;\;\mbox{whathever}\;\;\tau\in[1,\infty],
\end{equation}
which entails
\begin{equation}\label{boundoverlu}
\overline{u}\in L^\infty(\R^N);     
\end{equation} 
see \cite[Theorem 1.2]{CPT}. Let $\beta\in\R$ satisfy
\begin{equation}\label{defbeta}
\beta(p-1)+sp=N+\alpha-\gamma\frac{N-sp}{p-1}.
\end{equation}
One has $\beta<\frac{N}{p-1}$, because $\alpha<sp+\gamma\frac{N-sp}{p-1}$. So, Theorem \ref{fundsol} and Proposition \ref{decay} ensure that
$w(x):=|x|^{-\beta}\chi_{\R^N\setminus\overline{B}_1}(x)+u_1(x)\chi_{\overline{B}_1}(x)$ weakly solves the problem
\begin{equation*}
\left\{
\begin{alignedat}{2}
(-\Delta_p)^s w & =C(\beta)|x|^{-\beta(p-1)-sp}\;\; && \text{in}\;\;\R^N\setminus\overline{B}_1,\\
w & \equiv 1\;\; && \text{on}\;\;\overline{B}_1.
\end{alignedat}
\right.
\end{equation*}
If $\overline{w}:=Dw$, where $D>0$ complies with
$$D>\max\left\{\left(\frac{\hat c}{C(\beta)}\right)^{\frac{1}{p-1}},\sup_{B_1}\overline{u}\right\},$$
then, due to \eqref{auest} and \eqref{defbeta},
$$\overline{u}(x)\le\overline{w}(x)\;\;\text{in}\;\;\overline{B}_1,\quad
(-\Delta_p)^s\overline{u}
\le(-\Delta_p)^s\overline{w}\;\;\text{in}\;\;\R^N\setminus\overline{B}_1.$$
Now, Theorem \ref{weakcomp} yields $\overline{u}\le\overline{w}$ on the whole $\R^N$. Thus,
\begin{equation}\label{upperestu}
\overline{u}(x)\le D|x|^{-\beta}\;\;\forall\, x\in B_1^c.  
\end{equation}
Estimate \eqref{decayest} follows from \eqref{lowestu}, \eqref{defbeta}, and \eqref{upperestu}.
\end{proof}
\begin{rmk}
Obviously, $\overline{u}$ weakly solves the equation $(-\Delta_p)^s u=a(x) u^{-\gamma}$ on every open ball $B_{R}(x)$. By \cite[ Theorem 5.4]{IMS}, there exist $\alpha\in (0,1)$ and $C > 0$, not dependent on $x$, such that
\begin{equation}\label{osc_estimate}
\operatorname{osc}_{B_r(x)}\overline{u}
\leq C\left[\left(\|a\overline{u}^{-\gamma}\|_\infty R^{ps}\right)^{\frac{1}{p-1}}
+\|\overline{u}\|_{L^\infty(B_{R}(x))}+\operatorname{Tail}(\overline{u};x,R) \right]\left(\frac{r}{R}\right)^\alpha
\end{equation}
for all $r\in (0,R)$, where
$$\operatorname{Tail}(\overline{u};x,R):=\left( R^{ps}\int_{\R^N\setminus B_{R}(x)} \frac{|\overline{u}(y)|^{p-1}}{|x-y|^{N+ps}}\,\dy \right)^{\frac{1}{p-1}}.$$
Since \eqref{boundoverlu} entails
$$\operatorname{Tail}(\overline{u};x,R)\leq C_1\|\overline{u}\|_\infty,$$
using \eqref{osc_estimate} we achieve
$$\operatorname{osc}_{B_r(x)}\overline{u}\leq C_2
\left[\left(\|a\overline{u}^{-\gamma}\|_\infty R^{ps}\right)^{\frac{1}{p-1}}+ \|\overline{u}\|_\infty\right]\left(\frac{r}{R}\right)^\alpha,$$
with $C_2>0$ independent of $x$. Thus, $\overline{u}\in C^{0,\alpha}_{\mathrm{loc}}(\R^N)$.
\end{rmk}
\section{Existence of solutions}
Let $\overline{u}$ be the solution to \eqref{singprob} given by Theorem \ref{overusol}. Consider the truncated problem
\begin{equation}\label{trunprob}
\left\{
\begin{alignedat}{2}
(-\Delta_p)^s u & =\bar{f}(x,u)\;\; && \mbox{in}\;\;\R^N,\\
u & \ge0 && \mbox{in}\;\;\R^N,\\
%
\end{alignedat}
\right.
\end{equation}
where 
\begin{equation*}
\bar{f}(x,t):=f(x,\max\{\overline{u}(x),t\}),\;\; (x,t)\in\Omega\times\R.     
\end{equation*}
Thanks to \eqref{hypf} one has 
\begin{equation}\label{estbarf}
0\leq\bar{f}(\cdot,t)
\le a\,(\max\{\overline{u},t\})^{-\gamma}+a\,(\max\{\overline{u},t\})^r
\le a\,\overline{u}^{-\gamma}+a\,(\overline{u}^r+|t|^r)
\end{equation}
while the energy functional associated with \eqref{trunprob} is 
\begin{equation*}
\bar{J}(u):=
\frac{1}{p}\Vert u\Vert_{s,p}^p-\int_{\Omega}\Big(\int_0^{u(x)}\bar{f}(x,t)\dt\Big)\dx\quad
\forall\, u\in D^{s,p}(\R^N).
\end{equation*}
\begin{lemma}\label{tildeu}
Suppose \eqref{hypa} and \eqref{hypf} hold. Then $\bar{J}$ admits a minimizer $\tilde{u}\in D^{s,p}(\R^N)_+$, which weakly solves \eqref{trunprob}.
\end{lemma}
\begin{proof}
Since $r<p-1$, recalling \eqref{sumau}--\eqref{boundoverlu} and using \eqref{estbarf} we easily see that $\bar{J}$ turns out both well-defined and of class $C^1$. Moreover, H\"older's inequality, Young's inequality, Lemma \ref{weightedemb}, besides the continuous embedding $D^{s,p}(\R^N)\hookrightarrow L^{p^*}(\R^N)$, produce
\begin{equation*}
\begin{split}
\bar{J}(u) 
& \ge\frac{1}{p}\Vert u\Vert_{s,p}^p-\int_{\R^N} a\overline{u}^{-\gamma}|u|\dx
-\int_{\R^N} a\overline{u}^r|u|\dx-\frac{1}{r+1}\int_{\R^N}a|u|^{r+1}\dx\\
& \ge\frac{1}{p}\Vert u\Vert_{s,p}^p-\|u\|_{p^*}\|a\overline{u}^{-\gamma}\|_{(p^*)'}
-c_1\int_{\R^N} a\overline{u}^p\dx-c_2\int_{\R^N} a|u|^{\frac{p}{p-r}}\dx
-c_3\Vert u\Vert_{s,p}^{r+1}\\
& \ge\frac{1}{p}\Vert u\Vert_{s,p}^p-C\left(\|a \overline{u}^{-\gamma}\|_{(p^*)'}\Vert u\Vert_{s,p}-\Vert \overline{u}\Vert_{s,p}^p-\Vert u\Vert_{s,p}^{\frac{p}{p-r}}-\Vert u\Vert_{s,p}^{r+1}\right),
\;\; u\in D^{s,p}(\R^N),
\end{split}
\end{equation*}
with appropriate $C>0$. Therefore, $\bar{J}$ is coercive. Now, the Weierstrass-Tonelli theorem ensures that it attains its infimum, i.e., there exists $\tilde{u}\in D^{s,p}(\R^N)$ fulfilling
$$\bar{J}(\tilde{u})=\inf_{u\in D^{s,p}(\R^N)}\bar{J}(u).$$
Finally, from $\bar{f}(x,t)\ge 0$ in $\R^N\times\R$ it easily follows $\tilde{u}\ge0$. 
\end{proof}
\begin{thm}\label{exres}
Under conditions \eqref{hypa} and \eqref{hypf}, problem \eqref{prob} possesses a weak solution.
\end{thm}
\begin{proof}
Let $\bar{u}$ be as in Theorem \ref{overusol}, let $\tilde{u}$ be given by Lemma \ref{tildeu}, and let $\phi:=(\bar{u}-\tilde{u})^+$. Due to \eqref{hypf}, testing \eqref{singprob} and \eqref{trunprob} with $\varphi$ yields
$$\langle(-\Delta_p)^s\bar{u},\phi\rangle=
\int_{\R^N} a\,\bar{u}^{-\gamma}\phi\dx\le
\int_{\R^N} f(\cdot,\overline{u})\phi\dx=
\int_{\R^N}\bar{f}(\cdot,\tilde{u})\phi\dx=
\langle(-\Delta_p)^s\tilde{u},\phi\rangle,$$
namely
$$\langle(-\Delta_p)^s\bar{u}-(-\Delta_p)^s\tilde{u},
(\bar{u}-\tilde{u})^+\rangle\le 0.$$
By Lemma \ref{monfraclapl} this entails $\bar{u}\le\tilde{u}$. From \eqref{lowestu} and the properties of $\tilde{u}_1$ it next follows
\begin{equation}\label{ugreateroverlineu}
0<\tilde{u}_1\le\bar{u}\le\tilde{u}.
\end{equation}
Hence, $\tilde{u}$ turns out a positive weak solution to the equation $(-\Delta_p)^s u=f(x,u)$ in $\R^N$; c.f. Lemma \ref{tildeu}. Let us finally verify that $\displaystyle{\lim_{|x|\to\infty}}\tilde{u}(x)=0$. Combining \eqref{hypf} with \eqref{ugreateroverlineu} produces
$$0\le f(\cdot,\tilde{u})\le a\big(\tilde{u}^{-\gamma}+\tilde{u}^r\big)
\le a\big(\bar{u}^{-\gamma}+\tilde{u}^r\big)
\le a\big(\bar{u}^{-\gamma}+1\big)+a\,\tilde{u}^{p-1}.$$
Moreover, if
$$b_2(x):=a(x),\quad b_3(x):=a(x)\big(\bar{u}(x)^{-\gamma}+1\big) 
\quad\forall\, x\in\R^N$$
then $b_2\in L^1(\R^N)\cap L^\infty(\R^N)$ and, because of \eqref{sumau}, $b_3\in 
L^{\frac{p^*}{p-1}}(\R^N)\cap L^\infty(\R^N)$. Thus, through Theorem 1.2 in \cite{CPT} we obtain 
\begin{equation}\label{linftyu}
\tilde{u}\in L^{\infty} (\R^N).    
\end{equation}
Inequalities \eqref{lowestu} and \eqref{ugreateroverlineu} yield
\begin{equation}\label{lowefin}
\tilde{c}|x|^{-\frac{N-sp}{p-1}}\le\bar{u}(x)\le\tilde{u}(x),\;\; x\in B_1^c,
\end{equation}
while, thanks to \eqref{auest} and \eqref{linftyu},
\begin{equation*}
\begin{split}
f(x,\tilde{u}(x)) 
& \le a(x)\big(\tilde{u}(x)^{-\gamma}+\tilde{u}(x)^r\big) 
\le a(x)\big(\bar{u}(x)^{-\gamma}+\|\tilde{u}\|_{\infty}\big)\\
& \le\frac{c_a}{1+|x|^{N+\alpha}}\left(\tilde{c}^{-\gamma}|x|^{\gamma\frac{N-sp}{p-1}} +\|\tilde{u}\|_\infty\right)\le C|x|^{-\beta}
\;\;\forall\, x\in B_1^c,
\end{split}
\end{equation*}
where $\beta$ comes from \eqref{defbeta}. Proceeding as already done for \eqref{upperestu}, with $D>0$ bigger than $\|\tilde{u}\|_\infty$ too, we arrive at
\begin{equation*}
\tilde{u}(x)\le D|x|^{-\beta},\;\; x\in B_1^c,
\end{equation*}
whence, because of \eqref{lowefin}, the conclusion follows.
\end{proof}
\section*{Acknowledgments}
\noindent
The authors are members of the {\em Gruppo Nazionale per l'Analisi Matematica, la Proba\-bilit\`a e le loro Applicazioni} (GNAMPA) of the {\em Istituto Nazionale di Alta Matematica} (INdAM). \\
L. Gambera acknowledges the support of the project S.P.E.C.T.E.R. (Singular Problems, Ellipticity, and Convection Terms: Existence and Regularity), PRA 2020-2022, 
PIACERI, Linea 3, of the University of Catania.\\
The authors acknowledge the support of the project EdP.EReMo (Equazioni differenziali alle derivate parziali: esistenza, regolarità e molteplicità delle soluzioni), PRA 2022-2024, PIACE\-RI, Linea 1, of the University of Catania.

\end{document}